\documentclass[11pt,leqno]{amsart}
\usepackage{amssymb,amsmath,amsthm,amsfonts}
\usepackage{latexsym}
\usepackage{color}
\allowdisplaybreaks

\newcommand{\Hmm}[1]{\leavevmode{\marginpar{\tiny%
$\hbox to 0mm{\hspace*{-0.5mm}$\leftarrow$\hss}%
\vcenter{\vrule depth 0.1mm height 0.1mm width \the\marginparwidth}%
\hbox to 0mm{\hss$\rightarrow$\hspace*{-0.5mm}}$\\\relax\raggedright #1}}}
\newcommand{\nc}{\newcommand}
\nc{\les}{\lesssim}
\nc{\nit}{\noindent}
\nc{\nn}{\nonumber}
\nc{\D}{\partial}
\nc{\diff}[2]{\frac{d #1}{d #2}}
\nc{\diffn}[3]{\frac{d^{#3} #1}{d {#2}^{#3}}}
\nc{\pdiff}[2]{\frac{\partial #1}{\partial #2}}
\nc{\pdiffn}[3]{\frac{\partial^{#3} #1}{\partial{#2}^{#3}}}
\nc{\abs}[1] {\lvert #1 \rvert}
\nc{\cAc}{{\cal A}_c}
\nc{\cE}{{\cal E}}
\nc{\cF}{{\mathcal F}}
\nc{\cP}{{\cal P}}
\nc{\cV}{{\cal V}}
\nc{\cQ}{{\cal Q}}
\nc{\cGin}{{\cal G}_{\rm in}}
\nc{\cGout}{{\cal G}_{\rm out}}
\nc{\cO}{{\cal O}}
\nc{\Lav}{{\cal L}_{\rm av}}
\nc{\cL}{{\cal L}}
\nc{\cB}{{\cal B}}
\nc{\cZ}{{\cal Z}}
\nc{\cR}{{\cal R}}
\nc{\cT}{{\cal T}}
\nc{\cY}{{\cal Y}}
\nc{\cX}{{\cal X}}
\nc{\cXT}{{{\cal X}(T)}}
\nc{\cBT}{{{\cal B}(T)}}
\nc{\vD}{{\vec \mathcal{D}}}
\nc{\efield}{\mathcal{E}}
\nc{\vE}{{\vec \efield}}
\nc{\vB}{{\vec \mathcal{B}}}
\nc{\vH}{{\vec \mathcal{H}}}
\nc{\ty}{{\tilde y}}
\nc{\tu}{{\tilde u}}
\nc{\tV}{{\tilde V}}
\nc{\Pc}{{\bf P_c}}
\nc{\bx}{{\bf x}}
\nc{\bX}{{\bf X}}
\nc{\bXYZ}{{\bf XYZ}}
\nc{\bY}{{\bf Y}}
\nc{\bF}{{\bf F}}
\nc{\bS}{{\bf S}}
\nc{\dV}{{\delta V}}
\nc{\dE}{{\delta E}}
\nc{\TT}{{\Theta}}
\nc{\dPsi}{{\delta\Psi}}
\nc{\order}{{\cal O}}
\nc{\Rout}{R_{\rm out}}
\nc{\eplus}{e_+}
\nc{\eminus}{e_-}
\nc{\epm}{e_\pm}
\nc{\eps}{\varepsilon}
\nc{\vnabla}{{\vec\nabla}}
\nc{\G}{\Gamma}
\nc{\w}{\omega}
\nc{\mh}{h}
\nc{\mg}{g}
\nc{\vphi}{\varphi}
\nc{\tlambda}{\tilde\lambda}
\nc{\be}{\begin{equation}}
\nc{\ee}{\end{equation}}
\nc{\ba}{\begin{eqnarray}}
\nc{\ea}{\end{eqnarray}}

\nc{\g}{\gamma}
\nc{\ol}{\overline}

\newtheorem{theorem}{Theorem}[section]
\newtheorem{lemma}[theorem]{Lemma}
\newtheorem{prop}[theorem]{Proposition}
\newtheorem{corollary}[theorem]{Corollary}
\newtheorem{defin}[theorem]{Definition}

\nc{\pT}{\partial_T}
\nc{\pz}{\partial_z}
\nc{\pt}{\partial_t}
\nc{\la}{\langle}
\nc{\ra}{\rangle}
\nc{\infint}{\int_{-\infty}^{\infty}}
\nc{\halfwidth}{6.5cm}
\nc{\figwidth}{10cm}

\nc{\nlayers}{L} \nc{\nsectors}{M}
\nc{\indicator}{\mathbf{1}}
\nc{\Rhole}{R_{\rm hole}}
\nc{\Rring}{R_{\rm ring}}
\nc{\neff}{n_{\rm eff}}
\nc{\Frem}{F_{\rm rem}}
\nc{\Real}{\mathbb R}
\nc{\Z}{\mathbb Z}
\nc{\DD}{\Delta}
\nc{\cD}{\mathcal D}
\nc{\lnorm}{\left\|}
\nc{\rnorm}{\right\|}
\nc{\rnormp}{\right\|_{\ell^{p,\eps}}}
\nc{\rar}{\rightarrow}
\nc{\sgn}{{\rm sign}}
\allowdisplaybreaks
\date{\today}

\begin{document}

\title[Global smoothing for periodic KdV]{Global Smoothing for the Periodic KdV Evolution}

\author{M.~B.~Erdo\smash{\u{g}}an and  N.~Tzirakis}
\thanks{The authors were partially supported by NSF grants DMS-0900865 (B.~E.), and DMS-0901222 (N.~T.) }

\address{Department of Mathematics \\
University of Illinois \\
Urbana, IL 61801, U.S.A.}

\email{berdogan@uiuc.edu \\ tzirakis@math.uiuc.edu }

\begin{abstract}
The Korteweg-de Vries (KdV) equation with periodic boundary conditions is considered. It is shown that for $H^s$ initial data, $s>-1/2$, and for any $s_1<\min(3s+1,s+1)$, the difference of the nonlinear and linear evolutions is in $H^{s_1}$ for all times, with at most polynomially growing $H^{s_1}$ norm.
The result also extends to   KdV with a smooth, mean zero, time-dependent potential in the case $s\geq 0$. Our result and a theorem of Oskolkov for the Airy evolution imply that if one starts with   continuous and bounded variation initial data then the solution of KdV (given by the $L^2$ theory  of Bourgain) is a continuous function of space and time. In addition, we demonstrate smoothing for the modified KdV equation on the torus for $s>1/2$.
\end{abstract}

\maketitle

\section{Introduction}
In this paper we study the Korteweg de Vries (KdV) equation on the torus
\begin{equation}\label{eq:kdv}
\left\{
\begin{matrix}
u_{t}+u_{xxx} +2uu_{x}=0, & x \in {\mathbb T}, & t\in {\mathbb R},\\
u(x,0)=u_{0}(x)\in H^{s}({\mathbb T}),
\end{matrix}
\right.
\end{equation}
as well as the perturbed version with a smooth, mean-zero space-time potential
\begin{equation}\label{eq:potentialkdv}
\left\{
\begin{matrix}
u_{t}+u_{xxx} +(u^2+\lambda u)_{x}=0, & x \in {\mathbb T}, & t\in {\mathbb R},\\
u(x,0)=u_{0}(x)\in H^{s}({\mathbb T}).
\end{matrix}
\right.
\end{equation}
For both equations, we prove that the difference of the evolution with the Airy evolution is smoother than both the Airy and the nonlinear evolution.   This smoothing property is not apparent if one views the nonlinear evolution as a perturbation of the linear flow and apply standard Picard iteration techniques to absorb the nonlinear derivative term. The result will follow from a combination of the method of normal forms (through differentiation by parts) inspired by the result in \cite{bit} and the restricted norm method of Bourgain, \cite{Bou2}. As it is well-known, KdV is a completely integrable system with infinitely many conserved quantities. However, our method  in this paper do not rely on the integrability structure of KdV, and
thus it can be applied to other dispersive models. For example even if we perturb the KdV equation with a smooth function and break the integrability structure, we are still able to maintain the same smoothing properties. 

The history of the KdV equation is quite rich and the literature extensive. In this introduction we summarize some recent developments that are most relevant to our result. Note that all results were  proved in the subset of solutions that have mean zero. This assumption can easily be  dropped as we explain later in the paper since smooth solutions of KdV satisfy momentum conservation:
$$\int_{-\pi}^{\pi}u(x,t)dx=\int_{-\pi}^{\pi}u(x,0)dx.$$
To state some known results we start with a definition:
\begin{defin} Let $X$ be a Banach space. Starting with initial data $u_{0} \in H^{s}({\mathbb T})$, we say that the equation \eqref{eq:kdv} or \eqref{eq:potentialkdv} is locally well-posed, if there exists $T>0$ such that there exists a unique solution $u \in X \cap C_{t}^{0}H_{x}^{s}([0,T]\times \mathbb T)$. We also demand that there is continuity with respect to the initial data in the appropriate topology. If $T$ can be taken to be arbitrarily large then we say that the problem is globally well-posed. 
\end{defin} 

Well-posedness for nonsmooth data was first derived by Bourgain, \cite{Bou2}. He proved that the KdV equation is locally well-posed in $L^{2}(\mathbb T)$. The existence of local-in-time solutions for KdV was investigated in a class of function spaces $X$ and $Y$ that satisfy inequalities of the form
$$\|\partial_{x}(u^2)\|_Y \leq C \|u\|_{X}^2.$$ 
The definition of the spaces adopted to the KdV equation will be given later but the reader should keep in mind that these spaces will incorporate the dispersive character of the equation. 
The space-time Fourier transform of the linear solution is supported on the characteristic surface, $\tau=k^3$. Bourgain observed that if one first localizes in time, then  the Fourier transform of the nonlinear solution still concentrates near the characteristic surface due to the dispersive smoothing effect of $\partial_t-L$. Later, in \cite{kpv}, Kenig, Ponce, Vega extended the local theory to 
$H^s$, $s>-\frac12$, and proved that the estimate above fails for any $s<-\frac{1}{2}$.

For $L^2$ data these result can be extended to the case of arbitrary smooth perturbations, \cite{sta}. Due to energy conservation 
$$\int_{-\pi}^{\pi}u^{2}(x,t)dx=\int_{-\pi}^{\pi}u^{2}(x,0)dx$$ the KdV evolution is globally well-posed in $L^{2}(\mathbb T)$ and the solution is in $ C(\mathbb R;L^{2}(\mathbb T)$). Colliander, Keel, Staffilani, Takaoka, Tao, \cite{ckstt}, subsequently showed that KdV is globally well-posed in
$H^{s}(\mathbb T)$ for any $s \geq -\frac{1}{2}$ while adding a local well-posedness result for the endpoint $s=-\frac{1}{2}$. To extend  the local solutions globally in time they used the ``I-method'', developing a theory of almost conserved quantities starting with the energy. Although the initial data have infinite energy they showed that a smoothed out version of the solution cannot increase much in energy  going from one local-in-time interval to another. The iteration of this method leads to polynomial in time bounds for the rough Sobolev norm $H^s$, establishing the result. In \cite{kt},   Kappeler and   Topalov extended the latter result using the integrability properties of the equation and proved that   KdV admits global solutions in $H^{s}(\mathbb T)$ for any $s \geq -1$. 

Recently, in \cite{bit}, Babin, Ilyin, ant Titi gave a new proof of the $L^2$ theorem of Bourgain using normal form methods. Similar ideas were developed by Shatah \cite{sha}. This method was also used
in \cite{ETZ2} and \cite{ETZ3} to study near-linear behavior of certain dispersive models. They wrote the equation on the Fourier side and use differentiation by parts taking advantage of the large denominators that appear due to the dispersion. The method then proceeds  by obtaining estimates only on $L_{t}^{\infty}H_x^{s}$ norms. The clear advantage of their method apart from the simplicitly is that they can obtain (a property that was observed by Kwon and Oh in \cite{ko}) unconditional uniqueness for KdV in $C_t^{0}L_x^2$. In light of the theorem of   Christ, in \cite{chr2}, this result appears to be sharp. The drawback is that they cannot consider rougher than $L^2$ solutions and   they can only prove statements on negatively indexed Sobolev spaces under $L^2$ assumptions. The problem appears to be that  after two normal form transformations, the oscillations are killed,  and they are forced to work on $L_{t}^{\infty}H_x^{s}$, a space not optimal for proving multilinear estimates for low values of $s$. We will also use the differentiation by parts method in our proof, however, after one differentiation by parts we apply the restricted norm method of Bourgain that incorporates these oscillation in a nontrivial way. It is this combination that is enabling us to obtain a smoother solution starting with rough initial data for the full problem.

\subsection{Main Results}
\begin{theorem}\label{theo:main1}
Fix $s>-1/2$ and $s_1<\min(3s+1,s+1)$. Consider the real valued solution of KdV \eqref{eq:kdv}
on $\mathbb{T}\times \mathbb R$ with initial data $u(x,0)=g(x)\in H^s$. Assume that we have a growth bound $\|u(t)\|_{H^s}\leq C(\|g\|_{H^s}) (1+|t|)^{\alpha(s)}$. Then $u(t)-e^{tL}g\in C^0_tH^{s_1}_x$ and  
\[
\|u(t)- e^{tL}g\|_{H^{s_1}} \leq  C(s,s_1,\|g\|_{H^s}) \langle t\rangle^{1+6\alpha(s)},
\]
where $L=-\partial_x^3+\langle g\rangle \partial_x$.
\end{theorem}

{\bf Remark  1.} Note that for any $s>-1/3$ and $g\in H^s$, Theorem~\ref{theo:main1} and Theorem~\ref{thm:I2} below imply that
$$ u(t)- e^{tL}g \in L^2(\mathbb T)$$
for all times. We thus obtain a periodic analogue of Bourgain's theorem in \cite{Bou4}. In this paper he developed a method of high-low frequency decomposition of the initial data to obtain global solutions with infinite energy for the NLS equation on $\mathbb R^2$. In addition he proved that the nonlinear term is smoother than the initial data.  

{\bf Remark 2.} We note that, within the method of differentiation by parts, the range for $s_1$ when $s\leq 0$ in the theorem seems to be optimal up to the endpoint. This is because of the resonant term 
$$
\Big|\frac{u_k|u_k|^2}{k}\Big|=(|u_k||k|^s)^3 |k|^{-3s-1}
$$
appearing after the first differentiation by parts. For  general $H^s$ data, this term can not be in $H^{s_1}$ if $s_1>3s+1$. This also implies that for $s=-1/2$ there is no smoothing within the tools that we use.

{\bf Remark 3. } In an appendix, we demonstrate an analogous statement for the modified KdV equation on the torus.

\begin{theorem}\label{theo:main} Fix $s\geq 0$ and $s_1<  s+1$.
Consider \eqref{eq:potentialkdv} 
where $\lambda\in C^\infty(\mathbb T\times \mathbb R)$ is a mean-zero potential with bounded derivatives and initial data $u(x,0)=g(x)\in H^s$. Assume that we have a growth bound $\|u(t)\|_{H^s}\leq C(\|g\|_{H^s}) T(t)$ for some nondecreasing function $T$ on $[0,\infty)$.
Then $u(t)-e^{tL}g\in C^0_tH^{s_1}_x$ and  
\[
\|u(  t)-  e^{tL}g \|_{H^{s_1}} \leq  C(s,s_1,\|g\|_{H^s}) \langle t\rangle T(t)^9,
\]
where $L=-\partial_x^3+\langle g\rangle \partial_x$.
\end{theorem}

{\bf Remark 4. } We should note that the only reason that we are restricted to the range $s\geq 0$ in this theorem is the lack of global solutions at regularity levels below $L^2$.

{\bf Remark 5.} For $L^2$ initial data $g$, 
Theorem~\ref{theo:main} implies that $u-e^{tL}g\in C^0_tH^{1-}_x$, and hence is a continuous function of $x$ and $t$. 

Using this remark and the following theorem of Oskolkov, \cite{osk}, we obtain Corollary~\ref{cor:osk} below. We also note that using our theorem it is likely that other known properties of the Airy evolution could be extended to the KdV evolution.

\begin{theorem} \cite{osk} \label{oskolkov}
Let $L$ be as in the previous theorem and assume that $g$ is of bounded variation, then $e^{tL}g$ is a continuous function of $x$ if $t/2\pi$ is an irrational number. For rational values of $t/2\pi$, it is a bounded function with at most countably many discontinuities. 
Moreover, if $g$ is also continuous  then $e^{tL}g\in C^0_tC^0_x$.
\end{theorem}

\begin{corollary}\label{cor:osk}
Let $u$ be the real valued solution of \eqref{eq:potentialkdv} with  initial data $g \in BV\subset L^2$. Then,  $u$ is a continuous function of $x$ if $t/2\pi$ is an irrational number. For rational values of $t/2\pi$, it is a bounded function with at most countably many discontinuities. 
Moreover, if $g$ is also continuous  then $u\in C^0_tC^0_x$.
\end{corollary}

{\bf Remark 6.} Note that if $g\in H^s$ for some $s>1/2$, then $u\in C^0_tH^s_x \subset C^0_tC^0_x$ since KdV is globally well-posed. Oskolkov's theorem and smoothing allow us to have the same conclusion with initial data $g \in BV \cap C^0 \subset \bigcap_{s<1/2} H^s$.

\vspace{2mm}

The smoothing properties that we describe for  periodic KdV are in a way unexpected. The problem of smoothing for  KdV seems to be quite hard even if one poses the equation on the real line and as such has a long history. If we denote by $X^s$ and $Y^s$ the local theory spaces of the introduction at a regularity level that accommodates $s$ derivatives, a smoothing estimate for the KdV will read as
$$\|\partial_{x}(u^2)\|_{Y^{s_1}}\leq C \|u\|_{X^s}^2$$
with $s_1>s$. Unfortunately even for the real line this estimate fails as it was proved in \cite{cst}. In \cite{cst}, the autors were aiming at the existence of global-in-time solutions evolving from rougher than $L^2$ data. In the absence of a smoothing estimate they managed instead to prove a modified version of the estimate in the case that the functions in the nonlinear estimate had spatial Fourier transform supported away from zero. The failure of the estimate for wave interactions with low frequencies forced them to consider data in a space $H^{0,\alpha}$ for some $\alpha<0$ which suppresses low frequencies. This result didn't improve Bourgain's result  since this space fails to contain  $L^2$ for $\alpha<0$.  On the other hand, Colliander, Delort, Kenig and Stafillani in \cite{cdks}, obtained  bilinear smoothing estimates of the above form for NLS on $\mathbb R^2$ but with special nonlinearities. They used these estimates to obtain local well-posedness results for certain bilinear Schr\"odinger equations and establish polynomial in time bounds for the higher order Sobolev norms of nonintegrable NLS equations. 
  
For the periodic problem the first bilinear smoothing estimate that we know of is due to Kenig, Ponce and Vega, in \cite{kpv2}, for the 1d NLS on the torus and with the nonlinearity of the form $u^2$. In both cases (\cite{cdks}, \cite{kpv2}) the nonlinearities are not physical   (the equations do not conserve the $L^2$ norm), and they are useful in conjunction with other tools. We on the other hand prove a smoothing trilinear estimate for the full KdV nonlinearity after the normal form reduction. It is easy  to see that with the $uu_x$ nonlinearity the bilinear smoothing fails on the torus.  In \cite {chr},   Christ obtained a (local-in-time) smoothing estimate in the $\mathcal{FL}^p\to \mathcal{FL}^q$ setting for 1d cubic NLS. Thus, although the dynamics of periodic dispersive equations is a well studied subject, it appears that our result is the first smoothing estimate in Sobolev spaces for a physical partial differential equation.

We say a few words about the method of the proof. Following the argument in \cite{bit} and using a normal form reduction we first rewrite the equation in an equivalent but easier to deal with form. In this particular form the derivative in the nonlinearity is eliminated. The penalty one pays after such a reduction is  to increase the order of   the nonlinearity (in KdV from quadratic to cubic)  and to obtain resonant terms. Due to the absence of the zero Fourier modes (momentum conservation) the bilinear nonlinearity has no resonant terms. In addition we avoid the low frequency interactions that the authors in \cite{cst} faced. To estimate the new tri-linear term we now   decompose the nonlinearity into resonant and nonresonant terms.  It should be noted that   in the resonant terms the waves interact with no oscillation and hence they are always ``the enemy''. Any method to estimate nonlinear interactions of dispersive equations reaches its limits exactly due to these terms (see Remark 2 above). But it turns out that the nonsmooth resonant terms of the KdV cancel out and the gain of the derivative is more than enough to compensate for the remaining nonlinear terms. What is  more striking is the fact that the nonsmooth resonant terms cancel out even for the KdV equation with a potential. For the nonresonant terms, we apply the restricted norm method of Bourgain to the reduced nonlinearity to prove the smoothing.

\subsection{Notation}
To avoid the use of multiple constants, we  write $A \lesssim B$ to denote that there is an absolute  constant $C$ such that $A\leq CB$.  We also write $A\approx B$ to denote both $A\lesssim B$ and $B \lesssim A$. We define $\langle \cdot\rangle =1+|\cdot|$. We also reserve $\langle g\rangle$ notation for the avarage
of a $2\pi$-periodic function $g$.

We define the Fourier sequence of a $2\pi$-periodic $L^2$ function $u$ as
$$u_k=\frac1{2\pi}\int_0^{2\pi} u(x) e^{-ikx} dx, \,\,\,k\in \mathbb Z.$$
With this normalization we have
$$u(x)=\sum_ke^{ikx}u_k,\,\,\text{ and } (uv)_k=u_k*v_k=\sum_{m+n=k} u_nv_m.$$

As usual, for $s<0$, $H^{s}$ is the completion of $L^2$ under the norm 
$$\|u\|_{H^{s}}=\|\widehat u(k) (1+|k|^2)^{s/2}\|_{\ell^2}.$$ 
Note that for a mean-zero $L^2$ function $u$, $\|u\|_{H^{s}}\approx \|\widehat u(k) |k|^{s}\|_{\ell^2}$.
For a sequence $u_k$, with $u_0=0$, we will use $\|u\|_{H^{s}}$ notation to denote $\|u_k |k|^s\|_{\ell^2}$.

\vspace{.4cm}
\noindent
{\bf Acknowledgments.}
We thank J.~Bourgain, C.~E.~Kenig and V.~Zharnitsky for useful discussions. We also thank K.~Oskolkov for pointing out his result (Theorem~\ref{oskolkov}) to us.

\section{Some Results on KdV Evolution}
In this section we define the   spaces that are commonly used in the theory of periodic KdV, and list the results that our proof relies on.

We define the $X^{s,b}$ spaces for $2\pi$-periodic KdV via the norm
$$
\|u\|_{X^{s,b}}=\|\langle k\rangle^s \langle \tau-k^3\rangle^b \widehat u(k,\tau)\|_{L^2(dkd\tau)}.
$$
We also define the restricted norm
$$
\|u\|_{X^{s,b}_\delta}=\inf_{ \tilde u=u \text{ on } [-\delta,\delta]}\|\tilde u\|_{X^{s,b}}.
$$
The local well-posedness theory for periodic KdV was established in the space $X^{s,1/2}$.
Unfortunately, this space fails to control the $L^\infty_t H^s_x$ norm of the solution. To remedy this problem and ensure the continuity of KdV flow, the $Y^s$ and $Z^s$ spaces are defined in \cite{GTV} and \cite{ckstt}, based on the ideas of Bourgain, in  \cite{Bou2}, via the norms
$$
\|u\|_{Y^s}=\|u\|_{X^{s,1/2}}+ \| \langle k\rangle^s\widehat u(k,\tau) \|_{L^2(dk)L^1(d\tau)},
$$
$$
\|u\|_{Z^s}=\|u\|_{X^{s,-1/2}}+\Big\|\frac{\langle k\rangle^s\widehat u(k,\tau)}{\langle \tau-k^3\rangle}\Big\|_{L^2(dk)L^1(d\tau)}.
$$
One defines $Y^s_\delta$, $Z^s_\delta$ accordingly.
Note that if $u\in Y^s$ then $u\in L^\infty_t H^s_x$.

We use the following theorems from \cite{ckstt}:
\begin{theorem}\cite{ckstt} \label{thm:I1}
For any $s\geq -1/2$, the initial value problem \eqref{eq:kdv} is locally well-posed in $H^s$. In particular, $\exists \delta\approx \|g\|_{H^s}^{-3}$ such that there exists a unique solution
$$u\in C([-\delta,\delta];H^s_x(\mathbb T))\cap Y^s_\delta$$
with
$$\|u\|_{X^{s,1/2}_\delta}\leq \|u\|_{Y^s_\delta}\leq C \|g\|_{H^s}.$$
\end{theorem}
\begin{theorem}\cite{ckstt}\label{thm:I2}
For any $s\geq -1/2$, the initial value problem \eqref{eq:kdv} is globally well-posed in $H^s$. Moreover, for $-1/2\leq s < 0$,
$$\|u\|_{H^s}\leq C (1+|t|)^{-s+}$$
where $C$ depends on $\|g\|_{H^s}$.
\end{theorem}

The next theorem concerns  KdV with potential.

\begin{theorem}\label{thm:poten}
For any $s\geq 0$, the initial value problem \eqref{eq:potentialkdv} is locally well-posed in $H^s$. In particular, $\exists \delta\approx \|g\|_{H^s}^{-6}$ such that there exists a unique solution
$$u\in C([-\delta,\delta];H^s_x(\mathbb T))\cap Y^s_\delta$$
with
$$\|u\|_{X^{s,1/2}_\delta}\leq \|u\|_{Y^s_\delta}\leq C \|g\|_{H^s}.$$
\end{theorem}
The proof of this theorem follows the arguments in \cite{Bou2} and \cite{sta} which won't be repeated here. Notice that $\delta$ is different in this theorem since the equation is not scale invariant anymore. The dependence of $\delta$ on the size of the initial data given in the theorem is due to the following reason.
Bourgain, in \cite{Bou2}, proved that
$$\Big\|\psi_\delta(t)\int_0^t e^{L(t-r)} \partial_x(u v)(r) dr \Big\|_{Y^s}\lesssim \delta^{1/6}\|u\|_{X^{s,1/2}_\delta}\|v\|_{X^{s,1/2}_\delta}.
$$
Here $\psi_\delta(t):=\psi(t/\delta)$, where $\psi \in C^\infty$ and supported on $[-2,2]$, and $\psi(t)=1$ on $[-1,1]$.  
The gain in $\delta$ comes from the $L^4_{t,x}$ Strichartz estimate that he obtained. Using this inequality one has to pick $\delta\approx \|g\|_{H^s}^{-6}$ to close the contraction.

Now we show how one can extend the local solutions of   \eqref{eq:potentialkdv}  to   global ones.

\begin{theorem}\label{growth} The initial value problem \eqref{eq:potentialkdv} is globally well-posed in $L^2$. In particular, any $H^s$ norm for $s\geq 0$ grows at most exponentialy. Moreover, if $\partial_x\lambda\in L^1_tL^\infty_x$, then the $L^2$ norm remains bounded and any $H^s$ norm for $s>0$ grows at most polynomially. 
\end{theorem}
\begin{proof}
Note that using the equation we obtain
$$
\frac{d}{dt}\int_\mathbb{T}u^2 dx=2 \int_\mathbb{T} uu_t dx =\int_\mathbb{T} u^2  \lambda_x dx.
$$
Integrating in time we obtain
$$
\|u(t)\|_2^2\leq\|g\|_2^2+\int_0^t\|\lambda_x(s)\|_{L^\infty} \|u(s)\|_2^2 ds.
$$
Thus, by Gronval's inequality, we obtain
$$\|u\|_2\leq \|g\|_2 e^{\int_0^t\|\lambda_x(s)\|_{L^\infty}ds}.$$
Since $\lambda_x$ is bounded 
$$\|u\|_2\leq \|g\|_2 e^{Ct}.$$ 
Moreover if $\lambda_x\in L^1_tL^\infty_x$, then 
$$\|u\|_2\lesssim \|g\|_2.$$ 
In both cases, we can iterate the local solution and obtain a global-in-time solution evolving from an $L^2$ data.

To obtain the growth bound for the higher order norms we use Theorem~\ref{theo:main} repeatedly as follows. For $s\in(0,1)$, using the theorem with $s=0$, $s_1=s$, and $T(t)=e^{Ct}$ or $T(t)=C$ depending on the assumptions on $\lambda_x$, we obtain
$$\|u(t)-e^{tL}g\|_{H^{s}}\leq C_{\|g\|_2} \langle t\rangle T(t)^9.$$
Which implies by the unitarity of the linear evolution that
$$ \|u(t)\|_{H^{s}}\leq C_{\|g\|_2} \langle t\rangle T(t)^9.$$
One can continue iteratively with this process and reach any index $s$.
\end{proof}

We should note that, in \cite{bourgain, Bou2}, Bourgain studied  the growth of higher order Sobolev norms for KdV and the nonlinear Schr\"odinger equation (NLS). The critical observation was that if one has a  local   estimate of the form
$$\|u(t)\|_{H^s}\leq \|u_0\|_{H^s}+\|u_0\|_{H^s}^{1-\delta}$$
for $T=T(\|u_{0}\|_{H^1})$ then one can obtain
$$\|u(t)\|_{H^s}\leq C |t|^{\frac{1}{\delta}}.$$
This method was later refined by Staffilani in \cite{sta}. In particular, she proved that for the periodic KdV with a potential, the higher order Sobolev norms grow with a bound of the form $t^{(s-1)+}$ for any large $s$ assuming   an a priori bound on the $H^1$ norm. Staffilani's method
can also be applied in the case of time-dependent potentials\footnote{Personal communication with G.~Staffilani}, and for $s>1$ it gives a better polynomial growth bound than we obtained above, if one  knows that the $H^1$ norm remains bounded. Our result on the other hand gives some global bounds for any $s\geq 0$.

A final ingredient in our proof is the following Strichartz estimate of Bourgain \cite{Bou2}:
\begin{prop}\cite{Bou2} \label{prop:B}For any $\eps>0$ and $b>1/2$, we have
$$\|\chi_{[-\delta,\delta]}(t)u\|_{L^6_{t,x}(\mathbb R \times \mathbb T)}\leq C_{\eps,b} \|u\|_{X^{\eps,b}_\delta}.$$
\end{prop}

\section{Proofs of Theorem~\ref{theo:main1} and Theorem~\ref{theo:main}}\label{sec:evol}
We will give the proofs of these theorems simultaneously. This implies that Theorem~\ref{theo:main} would be valid also for $s\in(-1/2,0)$ if one knows the required well-possedness theory in this regularity level.

We will give the proof under the assumption that the initial data (and hence the solution at each time) is mean zero. To remove this assumption one changes the equation introducing two terms of the form $\langle g\rangle u_x$ and $\langle g \rangle \lambda_x$. The first term changes the linear operator from $-\partial_x^3$ to $L$ as it is stated in the theorem. We note that after this change the resonances and the multilinear estimates remain the same. The second term is in $H^{s}$ for any $s$, and in the calculations below  it  will only go into the $B$ operator defined below, which satisfies the same estimates.

Using the notation $u(x,t)=\sum_k u_k(t) e^{ikx}$ and $\lambda(x,t)=\sum_k\lambda_k(t)e^{ikx}$, we write  \eqref{eq:potentialkdv} on the Fourier side:
$$\partial_t u_k=ik^3u_k-ik  \sum_{k_1+k_2=k}(\lambda_{k_1}+u_{k_1})u_{k_2},\,\,\,\,\,\,u_k(0)=\widehat g(k),$$
Because of the mean zero assumption on $u$ and $\lambda$,  there are no zero harmonics in this equation.
Using the transformations
\begin{align*}
u_k(t)&=v_k(t)e^{ik^3t},\\
\lambda_k(t)&=\Lambda_k(t)e^{ik^3t},
\end{align*}
and the identity
$$(k_1+k_2)^3-k_1^3-k_2^3=3(k_1+k_2)k_1k_2,$$
the equation can be written in the form
\be\label{v_eq}
\partial_t v_k=-ik   \sum_{k_1+k_2=k}e^{-i3kk_1k_2t}(v_{k_1}+\Lambda_{k_1})v_{k_2}.
\ee

We start with the following proposition which follows from differention by parts. A similar proposition was proved in \cite{ETZ3}.
\begin{prop}\label{thm:dbp}
The system \eqref{v_eq} can be written in the following form:
\be\label{v_eq_dbp}
\partial_t[v+B(\Lambda+v,v)]_k=\rho_k+B(\partial_t\Lambda,v)_k+R(\Lambda+2v,\Lambda+v,v)_k,
\ee
where we define $B(u,v)_0=\rho_0=R(u,v,w)_0=0$, and for $k\neq 0$, we define
\begin{align*}
B(u,v)_k&= -\frac13\sum_{k_1+k_2=k}\frac{e^{-3ikk_1k_2t}  u_{k_1}v_{k_2}}{k_1k_2}
\\
\rho_k&=\frac{i}3\Lambda_k\sum_{|j|\neq|k|}\frac{\Lambda_j\overline{v_j}}{j}-\frac{i}3\frac{(\overline{\Lambda_k}+2\overline{v_k})(\Lambda_k+v_k)v_k}{k}\\
R(u,v,w)_k&=\frac{i}{3}\sum_{\stackrel{k_1+k_2+k_3=k}{(k_1+k_2)(k_1+k_3)(k_2+k_3)\neq 0}} \frac{e^{-3it(k_1+k_2)(k_2+k_3)(k_3+k_1)}}{k_1} u_{k_1}v_{k_2}w_{k_3}
\end{align*}
\end{prop}

\begin{proof}[Proof of Proposition~\ref{thm:dbp}]
Since $e^{-3ikk_1k_2t}=\partial_{t}( \frac{i}{3kk_1k_2}e^{-3ikk_1k_2t})$, using differentiation by parts we can rewrite \eqref{v_eq} as
\begin{align*}
\partial_{t}v_{k}&=\partial_{t}\Big(   \sum_{k_1+k_2=k}\frac{e^{-3ikk_1k_2t} (\Lambda_{k_1}+v_{k_1}) v_{k_2}}{3 k_1k_2}\Big)-  \sum_{k_1+k_2=k}\frac{e^{-3ikk_1k_2t}}{3k_1k_2}\partial_{t} [(\Lambda_{k_1}+v_{k_1}) v_{k_2}].
\end{align*}
Recalling the definition of  $B$, we can rewrite this equation in the form:
\be\label{eq:partialv}
\partial_t[v+B(\Lambda+v,v)]_k=B(\partial_t\Lambda,v)_k-\frac13\sum_{k_1+k_2=k}\frac{e^{-3ikk_1k_2t}}{k_1k_2} [v_{k_2}\partial_{t}v_{k_1}+(\Lambda_{k_1}+v_{k_1})\partial_t  v_{k_2}].
\ee
Note that since $v_0=0$, in the sums above $k_1$ and $k_2$ are not zero. Using \eqref{v_eq}, we have
\begin{align*}
&-\frac13\sum_{k_1+k_2=k}\frac{e^{-3ikk_1k_2t}}{k_1k_2}v_{k_2}\partial_{t}v_{k_1}  =\frac{i}3 \sum_{k=k_1+k_2}\frac{e^{-3ikk_1k_2t}v_{k_2}}{k_2}  \sum _{\mu+\nu=k_1\neq 0}e^{-3itk_1\mu\nu} (\Lambda_{\mu}+v_{\mu}) v_{\nu} \\
&= \frac{i}3\sum_{\stackrel{k=k_2+\mu+\nu}{\mu+\nu\neq 0}}\frac{e^{-3it[kk_2(\mu+\nu)+\mu\nu(\mu+\nu)]}}{k_2}v_{k_2} (\Lambda_{\mu}+v_{\mu})  v_{\nu}.
\end{align*} 
Using the identity
$$kk_2+\mu\nu=(k_2+\mu+\nu)k_2+\mu\nu=(k_2+\mu)(k_2+\nu)$$
and  by renaming the variables $k_2\to k_1$, $\mu\to k_2$, $\nu\to k_3$, we have that
\begin{multline}\label{reso1}
-\frac13 \sum_{k_1+k_2=k}\frac{e^{-3ikk_1k_2t}}{k_1k_2}v_{k_2}\partial_{t}v_{k_1} \\ 
=\frac{i}3\sum_{\stackrel{k_1+k_2+k_3=k}{k_2+k_3\neq 0}}\frac{e^{-3it(k_1+k_2)(k_2+k_3)(k_3+k_1)}}{k_1}v_{k_1} (\Lambda_{k_2}+v_{k_2}) v_{k_3}.
\end{multline}
Similarly,
\begin{multline}\label{reso2}
-\frac13 \sum_{k_1+k_2=k}\frac{e^{-3ikk_1k_2t}}{k_1k_2}(\Lambda_{k_1}+v_{k_1})\partial_{t}v_{k_2}  \\ =\frac{i}3\sum_{\stackrel{k_1+k_2+k_3=k}{k_2+k_3\neq 0}}\frac{e^{-3it(k_1+k_2)(k_2+k_3)(k_3+k_1)}}{k_1}(\Lambda_{k_1}+v_{k_1}) (\Lambda_{k_2}+v_{k_2}) v_{k_3}.
\end{multline}
Combining \eqref{reso1} and \eqref{reso2}, we can rewrite \eqref{eq:partialv} as
\begin{align*}
\partial_t[v+B(\Lambda+v,v)]_k&=B(\partial_t\Lambda,v)_k\\
&+\frac{i}3\sum_{\stackrel{k_1+k_2+k_3=k}{k_2+k_3\neq 0}}\frac{e^{-3it(k_1+k_2)(k_2+k_3)(k_3+k_1)}}{k_1}(\Lambda_{k_1}+2v_{k_1}) (\Lambda_{k_2}+v_{k_2}) v_{k_3}.
\end{align*}
Note that the set on which the phase on the right hand side vanishes is the disjoint union of the following  sets
\begin{align*}
S_{1}&=\{k_1+k_2=0\}\cap\{k_3+k_1=0\}\cap \{k_2+k_3\neq 0\}\Leftrightarrow \{k_1=-k,\ k_2=k,\ k_3=k\},\\
S_{2}&=\{k_1+k_2=0\} \cap \{ k_3+k_1\ne 0\} \cap \{k_2+k_3\neq 0\}\Leftrightarrow \{k_1=j,\ k_2=-j,\ k_3=k,\ |j| \neq |k| \},\\
S_{3}&=\{k_3+k_1=0\}\cap\{k_1+k_2\ne 0\}   \cap \{k_2+k_3\neq 0\}\Leftrightarrow \{k_1=j,\ k_2=k,\ k_3=-j,\ |j| \neq |k| \}.
\end{align*}
Thus, using the definition of $R(u,v,w)$, we have
\begin{align*}
\partial_t[v+B(\Lambda+v,v)]_k&=B(\partial_t\Lambda,v)_k+R(\Lambda+2v,\Lambda+v,v)_k\\
&+\frac{i}3\sum_{\ell=1}^{3}\sum_{S_{\ell}}\frac{(\Lambda_{k_1}+2v_{k_1}) (\Lambda_{k_2}+v_{k_2}) v_{k_3}}{k_1}.
\end{align*}
The proposition follows if we show that the second line above is equal to $\rho_k$. Note that
\begin{align}\label{reso3}
&\sum_{\ell=1}^{3}\sum_{S_\ell}\frac{(\Lambda_{k_1}+2v_{k_1}) (\Lambda_{k_2}+v_{k_2}) v_{k_3}}{k_1}=-\frac{(\Lambda_{-k}+2v_{-k}) (\Lambda_{k}+v_{k}) v_{k}}{k}\\\nonumber
&+v_k\sum_{|j|\neq |k|} \frac{(\Lambda_{j}+2v_{j}) (\Lambda_{-j}+v_{-j})}{j}+(\Lambda_{k}+v_{k})\sum_{|j|\neq |k|}\frac{(\Lambda_{j}+2v_{j})  v_{-j}}{j}
\end{align}
Note that using $v_j=\overline{v_{-j}}$ and $\Lambda_j=\overline{\Lambda_{-j}}$, we can rewrite the second line above as
\begin{align*}
&v_k\sum_{|j|\neq |k|} \frac{|\Lambda_{j}+v_{j}|^2+|v_j|^2+v_j\overline{\Lambda_{j}}}{j}+(\Lambda_{k}+v_{k})\sum_{|j|\neq |k|}\frac{(\Lambda_{j}\overline{v_j}+2|v_{j}|^2)}{j}\\
&=v_k\sum_{|j|\neq |k|} \frac{v_j\overline{\Lambda_{j}}}{j}+(\Lambda_{k}+v_{k})\sum_{|j|\neq |k|}\frac{\Lambda_{j}\overline{v_j}}{j}\\
&=2v_k\sum_{|j|\neq |k|} \frac{\Re(v_j\overline{\Lambda_{j}})}{j}+\Lambda_{k}\sum_{|j|\neq |k|}\frac{\Lambda_{j}\overline{v_j}}{j}.
\end{align*}
The first equality follows from the symmetry relation $j \leftrightarrow -j$. By the same token, the first summand in the last line above vanishes
since $\Re(v_j\overline{\Lambda_{j}})=\Re(\overline{v_j}\Lambda_{j})=\Re(v_{-j}\overline{\Lambda_{-j}}) $.
Using this in \eqref{reso3} we obtain
\begin{align*}
\sum_{\ell=1}^{3}\sum_{S_\ell}\frac{(\Lambda_{k_1}+2v_{k_1}) (\Lambda_{k_2}+v_{k_2}) v_{k_3}}{k_1}&=-\frac{(\Lambda_{-k}+2v_{-k}) (\Lambda_{k}+v_{k}) v_{k}}{k}+\Lambda_{k}\sum_{|j|\neq |k|}\frac{\Lambda_{j}\overline{v_j}}{j}\\
&=\frac3i \rho_k,
\end{align*}
which yields the assertion of the Proposition.
\end{proof}

Integrating \eqref{v_eq_dbp} from $0$ to $t$, we obtain
\begin{align*}
v_k(t)-v_k(0)&=-B(\Lambda+v,v)_k(t)+B(\Lambda+v,v)_k(0)+\int_0^t B(\partial_r\Lambda,v)_k(r) dr\\
&+ \int_0^t \rho_k(r) dr+\int_0^t R(\Lambda+2v,\Lambda+v,v)_k(r)dr.
\end{align*}
Transforming back to the $u$, $\lambda$ variables,  we have
\begin{align}\label{new_u}
u_k(t)-e^{ik^3t}u_k(0)&=-\mathcal{B}(\lambda+u,u)_k(t)+e^{ik^3t}\mathcal{B}(\lambda+u,u)_k(0)\\
&+\int_0^t e^{ik^3(t-r)}\mathcal{B}(e^{rL}\partial_r (e^{-rL}\lambda),u)_k(r) dr
+ \int_0^t e^{ik^3(t-r)}\tilde\rho_k(r) dr\nonumber\\
&+\int_0^t e^{ik^3(t-r)}\mathcal{R}(\lambda+2u,\lambda+u,u)_k(r)dr,\nonumber
\end{align}
where
\begin{align*}
\mathcal{B}(u,v)_k&=-\frac13\sum_{k_1+k_2=k}\frac{  u_{k_1} v_{k_2}}{k_1k_2},\\
\tilde\rho_k&=\frac{i}3\lambda_k\sum_{|j|\neq|k|}\frac{\lambda_j\overline{u_j}}{j}-\frac{i}3\frac{(\overline{\lambda_k}+2\overline{u_k})(\lambda_k+u_k)u_k}{k}\\
\mathcal{R}(u,v,w)_k&=\frac{i}{3}\sum_{\stackrel{k_1+k_2+k_3=k}{(k_2+k_3)(k_1+k_2)(k_1+k_3)\neq 0}} \frac{u_{k_1}v_{k_2}w_{k_3}}{k_1}.
\end{align*}
\begin{lemma}\label{apriori} For $s>-1/2$ and $s_1\leq s+1$, we have
$$\|\mathcal{B}(u,v)\|_{H^{s_1}}\lesssim \|u\|_{H^s}\|v\|_{H^s}.$$
For $s>-1/2$ and $s\leq s_1<3s+1$, we have
$$
\|\tilde\rho\|_{H^{s_1}}  \lesssim \|u\|_{H^s}\big(\|\lambda\|_{H^{s_1}}^2 +\|u\|_{H^s}^2\big).
$$
\end{lemma}
\begin{proof}
By symmetry we can assume in the estimate for   $\mathcal{B}(u,v)$ that $|k_1|\geq |k_2|$. Thus, for $s_1<s+1$ and $s>-1/2$, we have
\begin{align*}
\|\mathcal{B}(u,v)\|_{H^{s_1}} &\lesssim \Big\|\sum_{k_1+k_2=k,\,\,\,|k_1|\geq|k_2|}\frac{  |u_{k_1} v_{k_2}|  }{ |k_1k_2|}\Big\|_{H^{s_1}}
\\&\lesssim  \Big\|\sum_{k_1+k_2=k,\,\,\,|k_1|\geq|k_2|}\frac{ |k_1|^s |u_{k_1}| |k_2|^s |v_{k_2}| |k|^{s_1-s-1}}{|k_2|^{ s+1}}\Big\|_2\\
&\lesssim \Big\|\sum_{k_1+k_2=k,\,\,\,|k_1|\geq|k_2|}\frac{ |k_1|^s |u_{k_1}| |k_2|^s |v_{k_2}|  }{|k_2|^{s+1}}\Big\|_2\\
&\lesssim \Big\|\frac{ |k|^s v_{k}   }{|k|^{1+s}}\Big\|_{\ell^1} \big\||k|^s u_{k}\big\|_2
\lesssim  \big\|  |k|^s v_{k}\big\|_2  \big\||k|^{-1-s}\big\|_2 \|u\|_{H^s}\lesssim \|u\|_{H^s}\|v\|_{H^s}.
\end{align*}
In the last line we used Young's inequality and Cauchy-Schwarz.

Now note that for $s_1<3s+1$,
$$
\Big\|\frac{u_kv_kw_k}{k}\Big\|_{H^{s_1}}= \big\||u_kv_kw_k||k|^{3s} |k|^{s_1-3s-1}\big\|_{2}\lesssim
\big\||u_kv_kw_k||k|^{3s} \big\|_{2} \lesssim \|u\|_{H^s}\|v\|_{H^s}\|w\|_{H^s}.
$$
Also note that for any $-1/2\leq s$,
\begin{align*}
\Big\|\lambda_k\sum \frac{|\lambda_j||v_j|}{|j|}\Big\|_{H^{s_1}}&\leq \|\lambda\|_{H^{s_1}} \sum \frac{|\lambda_j|}{|j|^{1+s}}|u_j||j|^s\leq
\|\lambda\|_{H^{s_1}} \|\lambda\|_{H^{-s-1}}\|u\|_{H^s}\\
&\leq \|\lambda\|_{H^{s_1}}^2 \|u\|_{H^s}.
\end{align*}
The last two estimates imply the bound for $\tilde\rho$.
\end{proof}

Using the estimates in Lemma~\ref{apriori} in the equation \eqref{new_u}, we obtain (for $s>-1/2$ and $s_1<\min(3s+1,s+1)$)
\begin{align}\nonumber
&\|u(t)-e^{tL}g\|_{H^{s_1}}\lesssim \|\lambda(t)+u(t)\|_{H^s}\|u(t)\|_{H^s}+\|\lambda(0)+g\|_{H^s}\|g\|_{H^s} \\
&+\int_0^t \|e^{rL}\partial_r e^{-rL}\lambda(r)\|_{H^s}\|u(r)\|_{H^s}  dr+\int_0^t \|u(r)\|_{H^s}\big(\|\lambda(r)\|_{H^{s_1}}^2+\|u(r)\|_{H^s}^2\big)  dr  \nonumber \\
&+\Big\|\int_0^t e^{ik^3(t-r)}\mathcal{R}(\lambda+2u,\lambda+u,u)_k(r) dr\Big\|_{H^{s_1}} \nonumber\\
\label{udelta} &\lesssim \|u(t)\|_{H^s}+\|u(t)\|_{H^s}^2+\|g\|_{H^s}+\|g\|_{H^s}^2+\int_0^t \big(\|u\|_{H^s}+\|u\|_{H^s}^3 \big)dr \\
&+\Big\|\int_0^t e^{L(t-r)}\mathcal{R}(\lambda+2u,\lambda+u,u)(r) \, dr\Big\|_{H^{s_1}}, \nonumber
\end{align}
where the implicit constant in the second inequality depends on $\lambda\in C^\infty$ and
$$\mathcal{R}(u,v,w)(r,x)=\sum_{k\neq 0} \mathcal{R}( u,v,w)_k(r)  e^{ikx}.$$
Since our nonlinearity after differentiation by parts is not $uu_x$ anymore, we will be able to avoid the $Y^{s_1}$ and $Z^{s_1}$ spaces. Instead we will use the embedding $X^{s_1,b}\subset L^\infty_t H^{s_1}_x$ for $b>1/2$ and the following lemma from \cite{GTV}. Let $\psi_\delta(t):=\psi(t/\delta)$, where $\psi \in C^\infty$ and supported on $[-2,2]$, and $\psi(t)=1$ on $[-1,1]$.  
\begin{lemma} For $-\frac12<b^\prime\leq 0\leq b \leq b^\prime+1$, we have
\be\label{lem:xsb}
\Big\|\psi_\delta(t)\int_0^t e^{L(t-r)} F(r) dr \Big\|_{X^{s,b}}\lesssim \delta^{1-b+b^\prime} \|F\|_{X^{s,b^\prime}_\delta}.
\ee
\end{lemma}
For $t\in[-\delta/2,\delta/2]$, where $\delta$ is as in Theorem~\ref{thm:I1} or Theorem~\ref{thm:poten}, and $b>1/2$,
\begin{align} \label{udelta2}
&\Big\|\int_0^t e^{L(t-r)}\mathcal{R}(\lambda+2u,\lambda+u,u)(r)\, dr\Big\|_{H^{s_1}}\\ \nonumber& \leq\Big\|\psi_\delta(t)\int_0^t e^{L(t-r)}\mathcal{R}(\lambda+2u,\lambda+u,u)(r) \,dr\Big\|_{L^\infty_t H^{s_1}_x} \\\nonumber
& \lesssim \Big\|\psi_\delta(t)\int_0^t e^{L(t-r)}\mathcal{R}(\lambda+2u,\lambda+u,u)(r)\, dr\Big\|_{X^{s_1,b} }\\ \nonumber
&\lesssim \delta^{\eps/2} \|\mathcal{R}(\lambda+2u,\lambda+u,u)\|_{X^{s_1,-\frac12+\eps}_\delta},
\end{align}
for sufficiently small $\eps>0$.  
\begin{prop}\label{prop:nonlin} For $s>-1/2$, $s_1<\min(s+1,3s+1)$, and $\eps>0$ sufficiently small, we have
$$\|\mathcal{R}(u,v,w) \|_{X^{s_1,-\frac12+\eps}_\delta}\leq C \|u\|_{X^{s,1/2}_\delta} \|v\|_{X^{s,1/2}_\delta}\|w\|_{X^{s,1/2}_\delta}.$$
\end{prop}
We will prove this proposition later on. Using \eqref{udelta2} and the proposition above  in
\eqref{udelta}, we see that for $t\in[-\delta/2,\delta/2]$, we have (with implicit constant depending on $\lambda$)
\begin{align*}
\|u(t)-e^{tL}g\|_{H^{s_1}}&\lesssim \|u(t)\|_{H^s}+\|u(t)\|_{H^s}^2+\|g\|_{H^s}+\|g\|_{H^s}^2+\int_0^t \big(\|u\|_{H^s}+\|u\|_{H^s}^3\big) dr\\
&+\|u\|_{X^{s,1/2}_\delta}+\|u\|_{X^{s,1/2}_\delta}^3.
\end{align*}
In the rest of the proof the implicit constants also depend on $\|g\|_{H^s}$.
Fix $t$ large.  For $r\leq t$, we have the bound
$$\|u(r)\|_{H^s}\lesssim T(r)\leq T(t).$$
Thus, by the local theory, with $\delta \approx T(t)^{-6}$, we have
$$
\|u(j\delta)-e^{\delta L}u((j-1)\delta)\|_{H^{s_1}}\lesssim T(t)^{3},
$$
for any $j\leq t/\delta\approx t T(t)^6$. Here we used the local theory bound
$$
\|u\|_{X^{s,1/2}_{[(j-1)\delta,\,j\delta]}} \lesssim \|u((j-1)\delta)\|_{H^s}\lesssim T(t).
$$
Using this we obtain (with $J=t/\delta\approx t T(t)^6$)
\begin{align*}
\|u(J\delta)-e^{J \delta L}u(0)\|_{H^{s_1}}
&\leq \sum_{j=1}^J\|e^{(J-j) \delta L}u(j\delta)-e^{(J-j+1) \delta L}u((j-1)\delta)\|_{H^{s_1}}\\
&= \sum_{j=1}^J\| u(j\delta)-e^{ \delta L}u((j-1)\delta)\|_{H^{s_1}}\lesssim J T(t)^3 \approx t T(t)^9.
\end{align*}
This completes the proof of the growth bound stated in Theorem~\ref{theo:main}.

In the case of KdV without potential the local theory bound gives $\delta\approx T(t)^{-3}$ instead of the $T(t)^{-6}$ power. Also taking into account that $T(t)=\langle t\rangle^{\alpha(s)}$, we obtain the growth bound stated in Theorem~\ref{theo:main1}.

Now we will prove the continuity of $N(t):=u(t)-e^{tL}g$ in $H^{s_1}$. Using  \eqref{new_u}, we obtain
\begin{align}\label{NtNtau}
&N(t)-N(\tau)=\mathcal{B}(\lambda+u,u)_k(\tau)-\mathcal{B}(\lambda+u,u)_k(t)\\ \nonumber
&+(e^{ik^3t}-e^{ik^3\tau})\mathcal{B}(\lambda+u,u)_k(0)\\
&+\int_0^t e^{ik^3(t-r)}\mathcal{B}(e^{rL}\partial_r (e^{-rL}\lambda),u)_k(r) dr-\int_0^\tau e^{ik^3(\tau-r)}\mathcal{B}(e^{rL}\partial_r (e^{-rL}\lambda),u)_k(r) dr \nonumber\\ 
&+ \int_0^t e^{ik^3(t-r)}\tilde\rho_k(r) dr- \int_0^\tau e^{ik^3(\tau-r)}\tilde\rho_k(r) dr\nonumber\\
&+\int_0^t e^{ik^3(t-r)}\mathcal{R}(\lambda+2u,\lambda+u,u)_k(r)dr-\int_0^\tau e^{ik^3(\tau-r)}\mathcal{R}(\lambda+2u,\lambda+u,u)_k(r)dr.\nonumber
\end{align}
Fix $\tau$, we will show that the $H^{s_1}$ norm of each line in the formula above converges to zero as $t\to\tau$. We will assume that $t>\tau$ without loss of generality.
For the first line this follows  by using the difference  $u(\tau)-u(t)$, the continuity of the solution in $H^s$, and the a priori bounds for $\mathcal B$. For the second line, we use
the inequality (for any given $\eps\in [0,1]$) 
\be\label{exp}
\big| e^{ik^3t}-e^{ik^3\tau}\big|\lesssim \min(1,|k|^3|t-\tau|)\leq (|k|^3|t-\tau|)^\eps
\ee
and the a priori estimates for $\mathcal B$ in $H^{s_1+3\eps}$ for sufficiently small $\eps>0$, to obtain a bound of the form $|t-\tau|^\eps$. We now explain how to bound the fifth line.  The third and the forth lines can be treated similarly using $H^s$ norms instead of $X^{s,b}$ norms. We write the fifth line as   
\begin{align}\label{term1}
&\big(e^{ik^3(t-\tau)}-1\big)\int_0^\tau e^{ik^3(\tau-r)}\mathcal{R}(\lambda+2u,\lambda+u,u)_k(r)dr\\
&+\int_\tau^t e^{ik^3(t-r)}\mathcal{R}(\lambda+2u,\lambda+u,u)_k(r)dr.\label{term2}
\end{align}
To estimate $H^{s_1}$ norm of  \eqref{term1} we use \eqref{exp} with sufficiently small $\eps>0$ to obtain
\begin{align*}
\|\eqref{term1}\|_{H^{s_1}}\lesssim |t-\tau|^\eps \Big\|\int_0^\tau e^{ik^3(\tau-r)}\mathcal{R}(\lambda+2u,\lambda+u,u)_k(r)dr\Big\|_{H^{s_1+3\eps}}.
\end{align*}
To estimate the norm we divide the integral into $\tau/\delta$ pieces where $\delta$ is given by the local $X^{s,1/2}$ theory. Here $\delta$ depends on 
$\sup_{r\in[0,\tau]} \|u(r)\|_{H^s}$, which is finite due to global well-posedness. Then, we use
\eqref{udelta2} and Proposition~\ref{prop:nonlin} (with $s_1+3\eps$) to estimate each integral.
Finally, the bound for the $H^{s_1}$ norm of \eqref{term2} follows from the gain in $\delta$ in \eqref{udelta2} and Proposition~\ref{prop:nonlin}.

\section{Proof of Proposition~\ref{prop:nonlin}}
Recall that
$$\mathcal{R}(u,v,w)(r,x)=\sum_{k\neq 0} \mathcal{R}(u,v,w)_k(r)  e^{ikx}.$$
We need to prove  that
$$
\|\mathcal{R}(u,v,w)\|_{X^{s_1,-1/2+\eps}_\delta}\lesssim \|u\|_{X^{s,1/2}_\delta}\|v\|_{X^{s,1/2}_\delta}\|w\|_{X^{s,1/2}_\delta}.
$$
As usual this follows by considering the  $X^{s,b}$ norms instead of the restricted versions.
By duality it suffices to prove that
\begin{align}\label{dual}
\Big|\sum_k \int_{\mathbb R}\widehat{\mathcal R }(k,\tau)\widehat g(-k,-\tau)d\tau \Big|&=\Big|\int_{\mathbb R\times \mathbb T} \mathcal R(u,v,w) g \Big|\\
&\lesssim \|u\|_{X^{s,1/2}}\|v\|_{X^{s,1/2}}\|w\|_{X^{s,1/2}}\|g\|_{X^{-s_1,1/2-\eps}}. \nonumber
\end{align}
We note that
$$
\widehat{\mathcal{R}}(k,\tau)=\frac{i}{3}\int_{\tau_1+\tau_2+\tau_3=\tau}\sum_{\stackrel{k_1+k_2+k_3=k}{(k_2+k_3)(k_1+k_2)(k_1+k_3)\neq 0}} \frac{\hat u(k_1,\tau_1)\hat v(k_2,\tau_2)\hat w(k_3,\tau_3)}{k_1}.
$$
Let
\begin{align*}
f_1(k,\tau)&=|\widehat u(k,\tau)||k|^s \langle \tau-k^3\rangle^{1/2},\\
f_2(k,\tau)&=|\widehat v(k,\tau)||k|^s \langle \tau-k^3\rangle^{1/2},\\
f_3(k,\tau)&=|\widehat w(k,\tau)||k|^s \langle \tau-k^3\rangle^{1/2},\\
f_4(k,\tau)&=|\widehat g(k,\tau)| |k|^{-s_1} \langle \tau-k^3\rangle^{1/2-\eps}.
\end{align*}
Note that \eqref{dual} follows from
\be\label{dual1}
\sum_{\stackrel{k_1+k_2+k_3+k_4=0}{(k_2+k_3)(k_1+k_2)(k_1+k_3)\neq 0}} \int_{\tau_1+\tau_2+\tau_3+\tau_4=0}\frac{|k_1k_2k_3|^{-s}|k_4|^{s_1}\prod_{i=1}^4 f_i(k_i,\tau_i)}{|k_1|  \prod_{i=1}^4\langle\tau_i-k_i^3\rangle^{1/2-\eps}}\lesssim \prod_{i=1}^4\|f_i\|_2.
\ee
By Proposition~\ref{prop:B}, we have (for any $\eps>0$)
\be\label{L6}
\Big\|\Big(\frac{f_i |k|^{-\eps}}{\langle \tau-k^3\rangle^{1/2+\eps}}\Big)^\vee \Big\|_{L^6(\mathbb R \times \mathbb T)}\lesssim \|f_i\|_2.
\ee
Using $\tau_1+\tau_2+\tau_3+\tau_4=0$ and $k_1+k_2+k_3+k_4=0$, we have
$$
\sum_{i=1}^4 \tau_i-k_i^3=-k_1^3-k_2^3-k_3^3-k_4^3=3(k_1+k_2)(k_1+k_3)(k_2+k_3).
$$
Therefore
$$
\max_{i=1,2,3,4} \langle \tau_i-k_i^3 \rangle \gtrsim |k_1+k_2| |k_1+k_3| |k_2+k_3|.
$$
Since the inequality \eqref{dual1} is symmetric in $f_i$'s, it does not matter which of these terms is the maximum. Therefore without loss of generality we assume that
$$
\langle \tau_1-k_1^3\rangle =\max_{i=1,2,3,4} \langle \tau_i-k_i^3 \rangle  \gtrsim |k_1+k_2| |k_1+k_3| |k_2+k_3|.
$$
This implies that
\be\label{maxtau}
\prod_{i=1}^4\langle\tau_i-k_i^3\rangle^{1/2-\eps}\gtrsim (|k_1+k_2| |k_1+k_3| |k_2+k_3|)^{1/2-7\eps}\prod_{i=2}^4\langle\tau_i-k_i^3\rangle^{1/2+\eps}.
\ee
Also note that (since all factors are nonzero and $k_1+k_2+k_3+k_4=0$)
\be\label{kis}
|k_1+k_2| |k_1+k_3| |k_2+k_3|\gtrsim |k_i|,\,\,\,\,i=1,2,3,4.
\ee
Now we will prove that for $s>-1/2$, $s_1<\min(s+1,3s+1)$ and for $\eps$ sufficiently small,
\be\label{multiplier}
\frac{|k_1k_2k_3|^{-s}|k_4|^{s_1} }{|k_1| (|k_1+k_2| |k_1+k_3| |k_2+k_3|)^{1/2-7\eps}} \lesssim |k_1k_2k_3k_4|^{-\eps}.
\ee
By \eqref{kis}, this follows from
$$\frac{|k_1k_2k_3|^{-s}|k_4|^{s_1} }{|k_1| (|k_1+k_2| |k_1+k_3| |k_2+k_3|)^{1/2-11\eps}}\lesssim 1.
$$

First consider the case $s>-1/3 $, $s_1<\min(3s+1,s+1)$. Without loss of generality we can assume that $s_1\geq 0$.
Let $M=\max(|k_1|,|k_2|,|k_3|)$. Using $|k_1||k_1+k_2|\gtrsim|k_2|$ and $|k_1||k_1+k_3||k_3+k_2|\gtrsim|k_2|$, and by symmetry of $k_2, k_3$, we have
$$|k_1| (|k_1+k_2| |k_1+k_3| |k_2+k_3|)^{1/2-11\eps}\gtrsim M^{1-22\eps}.$$
Thus
$$\frac{|k_1k_2k_3|^{-s}|k_4|^{s_1} }{|k_1| (|k_1+k_2| |k_1+k_3| |k_2+k_3|)^{1/2-11\eps}}\lesssim
\frac{|k_1k_2k_3|^{-s}|k_4|^{s_1} }{M^{1-22\eps}}.
$$
Since $|k_1k_2k_3|^{-s}\leq M^{-3s}$ for $s<0$ and
$| k_1k_2k_3|^{-s}\leq M^{-s}$ for $s\geq 0$, we have $|k_1k_2k_3|^{-s}\lesssim M^{-\min(s,3s)}$.
Using this, the inequality $|k_4|\lesssim M$, and $0\leq s_1<\min(3s+1,s+1)$, we bound the multiplier by
$$ M^{-\min(s,3s)}M^{s_1-1+22\eps}\lesssim 1 \text{ for sufficiently small } \epsilon.
$$

Second consider the case $-1/2<s\leq-1/3$, $s_1<3s+1=\min(3s+1,s+1)\leq 0$. Using $|k_4|=|k_1+k_2+k_3|$
and $|k_1+k_2+k_3||k_2+k_3|\gtrsim |k_1|$, we have
\begin{multline} \nonumber
\frac{|k_1k_2k_3|^{-s}|k_4|^{s_1} }{|k_1| (|k_1+k_2| |k_1+k_3| |k_2+k_3|)^{1/2-11\eps}}\\
\lesssim \frac{|k_2k_3|^{-s}  }{|k_1|^{1+s-s_1} (|k_1+k_2| |k_1+k_3|)^{1/2-11\eps} |k_2+k_3|^{1/2+s_1-11\eps}}.
\end{multline}
Now using $|k_1||k_1+k_i|\gtrsim |k_i|$, we bound the multiplier by
$$
\frac{|k_2k_3|^{-s-\frac{1+s-s_1}{2}}  }{ (|k_1+k_2| |k_1+k_3|)^{\frac{s_1-s}{2}-11\eps} |k_2+k_3|^{1/2+s_1-11\eps}}  
\lesssim |k_2k_3|^{\frac{s_1-(3s+1)}{2}}  \lesssim 1.
$$

This finishes the proof of \eqref{multiplier}. Using \eqref{multiplier} and \eqref{maxtau} in \eqref{dual1} (and eliminating $|k_1|^{-\eps}$), we obtain
\begin{align*}
\eqref{dual1}&\lesssim   \sum_{k_1+k_2+k_3+k_4=0 } \int_{\tau_1+\tau_2+\tau_3+\tau_4=0}\frac{|k_2k_3k_4|^{-\eps} \prod_{i=1}^4 f_i(k_i,\tau_i)}{  \prod_{i=2}^4\langle\tau_i-k_i^3\rangle^{1/2+\eps}}.
\end{align*}
By Plancherel, and the convolution structure we can rewrite this as
\begin{align*}
\int_{\mathbb R \times \mathbb T}\widehat{f_1}(x,t) \prod_{i=2}^4\Big(\frac{f_i |k|^{-\eps}}{\langle \tau-k^3\rangle^{1/2+\eps}}\Big)^\vee(x,t) &\leq \|f_1\|_{L^2(\mathbb R \times \mathbb T)} \prod_{i=2}^4\Big\|\Big(\frac{f_i |k|^{-\eps}}{\langle \tau-k^3\rangle^{1/2+\eps}}\Big)^\vee \Big\|_{L^6(\mathbb R \times \mathbb T)}\\
&\lesssim \prod_{i=1}^4\|f_i\|_2.
\end{align*}
In the last inequality we used \eqref{L6}.

\section{Appendix: Smoothing for Modified KdV}
In this section we consider the modified KdV equation (mKdV) in the form 
\begin{align}\label{mkdv}
&v_t+v_{xxx}=6v^2v_x\\\nonumber
&v(x,0)=g(x)\in H^s(\mathbb T),\,\,\,\,\,\,s>1/2.
\end{align}
The mKdV equation satisfies both mean and $L^2$ conservation. 
Set $\mu=\langle v^2\rangle=\langle g^2\rangle$. Note that $w(x,t)=v(x-6\mu t,t)$ satisfies 
\begin{align}\label{mkdv1}
&w_t+w_{xxx}=6(w^2-\langle w^2\rangle) w_x\\ \nonumber
&w(x,0)=g(x)\in H^s(\mathbb T),\,\,\,\,\,\,s>1/2.
\end{align}
\begin{theorem}
\label{theo:mkdv}
Fix $s>1/2$ and $s_1<\min(3s-1,s+1)$. Consider the real valued solution of mKdV \eqref{mkdv}
on $\mathbb{T}\times \mathbb R$ with mean-zero initial data $v(x,0)=g(x)\in H^s$.   Then $v(t)-e^{tL}g\in C^0_t(\mathbb R;H^{s_1}_x)$,
where $L=-\partial_x^3-6\langle g^2\rangle \partial_x$.
\end{theorem}
\begin{proof}
Let $w$ be a solution of \eqref{mkdv1} with a mean-zero initial data $g\in H^s(\mathbb T)$, $s>1/2$.
We will use the Miura transform $\mathbf M w:=w_x+w^2-\langle w^2\rangle$. Following \cite{ckstt2}, we note that if $w$ solves \eqref{mkdv1} with a mean-zero initial data, then $u=\mathbf{M}w$ solves the KdV equation
$$
u_t+u_{xxx}=6u u_x
$$
with mean-zero initial data $u(x,0)=\mathbf Mg(x)$. 
The following theorem was proved in \cite{ckstt2} (here $H^s_0(\mathbb T)=\{u\in H^s(\mathbb T): \langle u\rangle=0\}$):
\begin{theorem}\cite{ckstt2} Let $s \geq 1/2$. Then the map $\mathbf M$ is a locally Lipschitz bijection from $H^s_0(\mathbb T)$ to $H^{s-1}_0(\mathbb T)$, and the inverse map $\mathbf M^{-1}$ is   locally Lipschitz   from $H^{s-1}_0(\mathbb T)$ to $H^s_0(\mathbb T)$.
\end{theorem}
This implies in particular that $\mathbf M$ maps $C^0_t(\mathbb R;H^{s}_0(\mathbb T))$ to 
$C^0_t(\mathbb R;H^{s-1}_0(\mathbb T))$, and $\mathbf M^{-1}$ maps  $C^0_t(\mathbb R;H^{s-1}_0(\mathbb T))$ to $C^0_t(\mathbb R;H^{s }_0(\mathbb T))$. The theorem will follow from this and Theorem~\ref{theo:main1}. Indeed, by Theorem~\ref{theo:main1}, $u=\mathbf Mw$ satisfies
$$
u(t)-e^{-\partial_x^3 t}\mathbf Mg \in C^0_t(\mathbb R;H^{\rho}_0(\mathbb T))
$$ 
for any $\rho<\min(3s-2,s)$. This implies that
\be\label{Mu1}
\mathbf M^{-1}u(t)-\mathbf M^{-1}e^{-\partial_x^3 t}\mathbf Mg \in C^0_t(\mathbb R;H^{s_1}_0(\mathbb T))
\ee
for any $s_1<\min(3s+1,s+1)$. In addition, since $\partial_x$ and $e^{-\partial_x^3 t}$ commutes, we have 
\begin{align*}
e^{-\partial_x^3 t}\mathbf Mg-\mathbf M e^{-\partial_x^3 t}g &= (e^{-\partial_x^3 t} g)^2 -e^{-\partial_x^3 t}(g^2) -\langle (e^{-\partial_x^3 t}g)^2\rangle+\langle g^2\rangle\\
&=(e^{-\partial_x^3 t} g)^2 -e^{-\partial_x^3 t}(g^2)\in C^0_t(\mathbb R;H^{s}_0(\mathbb T)),
\end{align*}
by the algebra property of $H^s$, $s>1/2$. Therefore,
\be\label{Mu2}
\mathbf M^{-1}e^{-\partial_x^3 t}\mathbf Mg-  e^{-\partial_x^3 t}g  \in C^0_t(\mathbb R;H^{s+1}_0(\mathbb T)).
\ee
Writing
$$
w(t)-e^{-\partial_x^3 t}g=\mathbf M^{-1} u(t) - \mathbf M^{-1}e^{-\partial_x^3 t}\mathbf Mg + \mathbf M^{-1}e^{-\partial_x^3 t}\mathbf Mg-e^{-\partial_x^3 t}g,
$$
and using \eqref{Mu1} and \eqref{Mu2}, we obtain 
$$w(t)-e^{-\partial_x^3 t}g\in C^0_t(\mathbb R;H^{s_1}_0(\mathbb T)),\,\,\,\,\,s_1<\min(3s+1,s+1).$$
To obtain the result for $v$, let $F(x,t):=v(x,t)-(e^{Lt}g)(x)$. Note that
$$
F(x-6\mu t,t)=w(x,t)-(e^{Lt}g)(x-6\mu t)=w(x,t)-(e^{-\partial_x^3t}g)(x)\in C^0_t(\mathbb R; H^{s_1}_0(\mathbb T)),
$$
for $s_1<\min(3s+1,s+1)$. This yields the statement for $v$ by translation invariance of Sobolev spaces.
\end{proof}

\end{document}